\documentclass[11pt, twoside]{amsart}

\title[Division algebras in monoidal categories]{Division algebras in monoidal categories}

\author[J. Kesten]{Jacob Kesten}
\address{Kesten: Department of Mathematics, Rice University,
P.O. Box 1892, Houston, TX 77005, USA}
\email{jgk3@rice.edu}

\author[C. Walton]{Chelsea Walton}
\address{Walton: Department of Mathematics, Rice University,
P.O. Box 1892, Houston, TX 77005, USA}
\email{notlaw@rice.edu}

\thanks{CW was partially supported by the NSF Grant DMS-2348833, and an AMS Claytor-Gilmer Research Fellowship.}

\usepackage{amsmath}
\usepackage{amsfonts}
\usepackage{amsthm}
\usepackage{amssymb,bbm}
\usepackage{dsfont}
\usepackage{stmaryrd}
\usepackage{mathabx}
\usepackage{appendix}
\usepackage[english]{babel}
\usepackage{url}
\usepackage{fancyhdr}
\usepackage{graphicx}
\usepackage{verbatim}
\usepackage[normalem]{ulem}
\usepackage[shortlabels]{enumitem}
\newcommand{\stkout}[1]{\ifmmode\text{\sout{\ensuremath{#1}}}\else\sout{#1}\fi}
\usepackage[
colorlinks=true,
linkcolor=black, 
anchorcolor=black,
citecolor=black,
urlcolor=black, 
]{hyperref}
\usepackage[all,cmtip]{xy}
\usepackage{geometry}
\usepackage[dvipsnames]{xcolor}
\usepackage{quiver}
\usepackage{lscape} 
\usepackage{multicol}
\usepackage{tensor}
\usepackage[T1]{fontenc}
\usepackage{stmaryrd}
\usepackage{mathbbol}
\usepackage{mathtools}
\usepackage{changepage}

\usepackage[utf8]{inputenc}
\usepackage{tikz}
\usepackage{tikz-cd}

\allowdisplaybreaks

\newcommand \blue {\textcolor{blue}}

\definecolor{forest}{rgb}{0.0, 0.55, 0.0}
\definecolor{BetterBlue}{rgb}{0.2, 0.4, 1.0}
\definecolor{BetterPurple}{rgb}{0.6, 0.3, 0.9}

\usepackage{enumitem}

\usepackage[notcite,notref,final]{showkeys}

\usepackage{calrsfs}
\DeclareMathAlphabet{\cal}{OMS}{zplm}{m}{n}

\DeclareMathAlphabet{\mathsf}{OT1}{cmss}{m}{n} 

\newcommand{\rdual}[1]{\tensor*[^*]{#1}{}}
     
\newcommand{\varep}{\varepsilon}

\newcommand{\Id}{\textnormal{Id}}
\newcommand{\id}{\textnormal{id}}
\newcommand{\one}{\mathbbm{1}}

\newcommand{\Hom}{\textnormal{Hom}}

\newcommand{\actL}{\triangleright}
\newcommand{\actR}{\triangleleft}

\newcommand\natisom{\stackrel{\hbox{$\sim$\hspace{.02in}}}{\smash{\Rightarrow}\rule{0pt}{0.4ex}}}

\newcommand\equivto{\stackrel{\hbox{$\sim$\hspace{.02in}}}{\smash{\to}\rule{0pt}{0.3ex}}}

\newcommand{\Alg}{\mathsf{Alg}}

\newcommand{\Bimod}{\mathsf{Bimod}}

\newcommand{\End}{\textnormal{End}}

\newcommand{\Mod}{\mathsf{Mod}}

\newcommand{\mA}{\mathbb{A}}

\newcommand{\cA}{\cal{A}}
\newcommand{\cC}{\cal{C}}

\newcommand{\cB}{\cal{B}}

\newcommand{\cM}{\cal{M}}
\newcommand{\cN}{\cal{N}}

\numberwithin{equation}{section}

\newtheorem{theorem}{Theorem}[section]

\newtheorem{proposition}[theorem]{Proposition}

\newtheorem{lemma}[theorem]{Lemma}

\newtheorem{theorem*}{Theorem}

\theoremstyle{definition}
\newtheorem{definition}[theorem]{Definition}

\newtheorem{example}[theorem]{Example}
\newtheorem{remark}[theorem]{Remark}

\makeatletter              
\let\c@equation\c@theorem  
\makeatother
\numberwithin{equation}{section}

\makeatletter
\@namedef{subjclassname@2020}{%
  \textup{2020} Mathematics Subject Classification}
\makeatother


\tikzcdset{scale cd/.style={every label/.append style={scale=#1},
    cells={nodes={scale=#1}}}}

\begin{document}

\begin{abstract}
This work adapts the equivalent definitions of division algebras over a field into multiple types of division algebras in a monoidal category. Examples and consequences of these definitions are then established in various monoidal settings. 
\end{abstract}

\maketitle

\setcounter{tocdepth}{3}



\section{Introduction}\label{sec:intro}

The goal of this work is to introduce the theory of division algebras over a field to  the monoidal setting.
Before continuing, note that linear structures here are over an algebraically closed field $\Bbbk$ of characteristic zero. Categories $\cC$ are assumed to be locally small, unless stated otherwise.
We begin by recalling the definition of division algebras over $\Bbbk$.

\begin{definition}\label{def:DA}
    Let $A$ be a non-zero, associative, unital $\Bbbk$-algebra. We say that $A$ is a \textit{division algebra over $\Bbbk$} if it satisfies any, and hence all, of the following equivalent conditions.
    \begin{enumerate}\renewcommand{\labelenumi}{(\roman{enumi})}
        \item Every non-zero element of $A$ is left (or right) invertible.
        \smallskip
        \item Every left (or right) $A$-module is free.
        \smallskip
        \item The regular left (or right) $A$-module is a simple module.
    \end{enumerate}
\end{definition}

Prior work that adapts this definition to the monoidal setting includes the following. Motivated by the Artin-Wedderburn Theorem, \cite{KongZheng} introduced division algebras in a multifusion category by adapting the right version of Definition~\ref{def:DA}(iii). Division algebras in fusion categories were also introduced in \cite{Grossman} and \cite{GrossmanSnyder} as tools for studying Morita equivalence classes and autoequivalences; there, the left version of Definition~\ref{def:DA}(iii) is adapted. 
Here, we generalize  Definition~\ref{def:DA}(ii,iii) for the monoidal setting, and also create a monad-theoretic description of division algebras.

\begin{definition}[Definitions~\ref{def:SimpDA}, \ref{def:EssDA}, \ref{def:monadic}] \label{def:DAmonoidal} Take $\cC$ to be an abelian monoidal category.
\begin{enumerate}\renewcommand{\labelenumi}{(\roman{enumi})}
    \item A non-zero algebra in $\cC$ is called a \textit{right (left) monadic division algebra} if its associated ``tensor on the right (left) monad'' has equivalent Kleisli and Eilenberg-Moore categories.
    \smallskip
    \item A non-zero algebra in $\cC$ is called a \textit{right (left) essential division algebra} if the right (left) free-module functor is essentially surjective.
    \smallskip
    \item A non-zero algebra is called a \textit{right (left) simplistic division algebra} if the right (left) regular module is simple.
\end{enumerate}
\end{definition}

For example, when $\cC$ is the monoidal category of $\Bbbk$-vector spaces, part~(ii) (resp., part~(iii)) of the definition above recovers part~(ii) (resp., part~(iii)) of Definition~\ref{def:DA}.
Our main result explores how these new definitions relate in a variety of monoidal settings.

\begin{theorem}[Props.~\ref{Prop:ess to simp}, \ref{prop:LiffR}, \ref{Prop:MonEss}]\label{THM} Let $A$ be a non-zero algebra in an abelian monoidal category~$\cC$. 
    \begin{enumerate}\renewcommand{\labelenumi}{\textnormal{(\roman{enumi})}}
        \item $A$ is a monadic division algebra in $\cC$ precisely when it is an essential division algebra in $\cC$.
        \smallskip
        \item Suppose that $\cC$ is rigid with simple unit, and $A$ has a simple module in $\cC$. If $A$ is an essential division algebra, then $A$ is a simplistic division algebra.
         \smallskip
        \item When $\cC$ is a pivotal multifusion category, then each left version of a  division algebra in $\cC$ is equivalent to its right version in $\cC$.
    \end{enumerate}
\end{theorem}

Note that the hypothesis on $A$ in part (ii) holds in many settings including in semisimple categories [Remark~\ref{rem:hypsimple}]. Also, the condition that $\cC$ is abelian is only necessary when working with simplistic division algebras. Results on essential and monadic division algebras can be given in not-necessarily-abelian monoidal categories by replacing the non-zero condition for algebras with the condition that the algebra admits more than one isoclass of modules in $\cC$; see Lemma~\ref{lem:zeromod}. 

\smallskip

Throughout the paper, we also supply several examples for the types of division algebras in Definition~\ref{def:DAmonoidal}, especially to show how the types differ in various monoidal categories. 

\smallskip

\begin{itemize}[\footnotesize{$\bullet$}]
\item We provide a sufficient condition for a monad $T$ on a monoidal category $(\cC, \otimes, \one)$ to ensure that $T(\one)$ is a monadic division algebra in $\cC$ [Proposition~\ref{Prop:adjtriv to Ess}]. Using this, we produce a monadic division algebra in the non-abelian monoidal category $\mathsf{Set}$  [Example~\ref{ex4.2}].

 \smallskip
 
 \item For certain semisimple, rigid, abelian monoidal categories with simple unit, we show that simplistic $\not \Rightarrow$ essential [Examples~\ref{ex:FIB},~\ref{ex:grouptype}]; cf. Theorem~\ref{THM}(ii).

\smallskip

\item For a monoidal category $\cC$  with non-simple unit, we show that the unit is an essential division algebra in $\cC$ that is not a simplistic division algebra in $\cC$ [Example~\ref{ex:one}]; cf. Theorem~\ref{THM}(ii).

\smallskip
 
 \item In rigid categories $\cC$, take the algebras  $X \otimes X^*$ and ${}^* X \otimes X$, for $X \in \cC$, with structure morphisms given by (co)evaluation morphisms. These are simplistic division algebras in $\cC$ precisely when $X$ is a simple object in $\cC$, and are essential division algebras in $\cC$ precisely when  $X$ is a one-sided invertible object in $\cC$ [Proposition~\ref{prop:IntEndDA}]. 
 \end{itemize}

\smallskip

Regarding the last item, one-sided invertibility implies simplicity under certain conditions on $\cC$ [Lemma~\ref{lem:invsimple}], so this item illustrates Theorem~\ref{THM}(ii) (see Remark~\ref{rem:XX*}). The algebras in the last item are also examples of internal End algebras, and the result there holds in settings where Ostrik's Theorem [Theorem~\ref{Thm:Ostrik}] is valid (e.g., in multifusion  categories); see Lemma~\ref{lem:invertible}, Propositions~\ref{Prop:simpMF},~\ref{Prop:essMF}. This yields more examples of simplistic and essential division algebras in  monoidal categories.


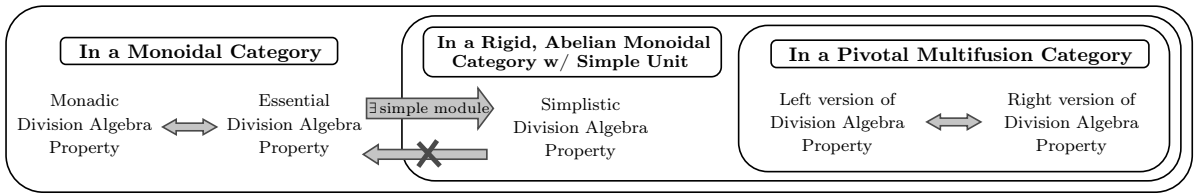
\begin{figure}[h!]
\scalebox{0.82}{
\tikzset{every picture/.style={line width=0.75pt}} 
\begin{tikzpicture}[x=0.75pt,y=0.75pt,yscale=-1,xscale=1]

\draw  [fill={rgb, 255:red, 255; green, 255; blue, 255 }  ][line width=0.75]  (34.5,40.08) .. controls (34.5,27.43) and (44.76,17.18) .. (57.4,17.18) -- (741.19,17.18) .. controls (753.84,17.18) and (764.09,27.43) .. (764.09,40.08) -- (764.09,108.78) .. controls (764.09,121.43) and (753.84,131.68) .. (741.19,131.68) -- (57.4,131.68) .. controls (44.76,131.68) and (34.5,121.43) .. (34.5,108.78) -- cycle ;
\draw  [fill={rgb, 255:red, 255; green, 255; blue, 255 }  ][line width=0.75]  (278,43.48) .. controls (278,32.27) and (287.09,23.18) .. (298.3,23.18) -- (736.79,23.18) .. controls (748,23.18) and (757.09,32.27) .. (757.09,43.48) -- (757.09,104.38) .. controls (757.09,115.59) and (748,124.68) .. (736.79,124.68) -- (298.3,124.68) .. controls (287.09,124.68) and (278,115.59) .. (278,104.38) -- cycle ;
\draw  [fill={rgb, 255:red, 255; green, 255; blue, 255 } ][line width=0.75]  (485.09,47.02) .. controls (485.09,37.07) and (493.15,29.02) .. (503.09,29.02) -- (732.09,29.02) .. controls (742.03,29.02) and (750.09,37.07) .. (750.09,47.02) -- (750.09,101.02) .. controls (750.09,110.96) and (742.03,119.02) .. (732.09,119.02) -- (503.09,119.02) .. controls (493.15,119.02) and (485.09,110.96) .. (485.09,101.02) -- cycle ;
\draw  [color={rgb, 255:red, 74; green, 74; blue, 74 }  ,draw opacity=1 ][fill={rgb, 255:red, 197; green, 197; blue, 197 }  ,fill opacity=1 ][line width=0.75]  (130.84,91.53) -- (139.14,86.84) -- (139.14,89.19) -- (155.75,89.19) -- (155.75,86.84) -- (164.05,91.53) -- (155.75,96.23) -- (155.75,93.88) -- (139.14,93.88) -- (139.14,96.23) -- cycle ;
\draw   (68.01,39.8) .. controls (68.01,37.98) and (69.49,36.5) .. (71.31,36.5) -- (235.11,36.5) .. controls (236.93,36.5) and (238.41,37.98) .. (238.41,39.8) -- (238.41,49.71) .. controls (238.41,51.53) and (236.93,53.01) .. (235.11,53.01) -- (71.31,53.01) .. controls (69.49,53.01) and (68.01,51.53) .. (68.01,49.71) -- cycle ;
\draw   (291.21,36.22) .. controls (291.21,32.85) and (293.94,30.12) .. (297.31,30.12) -- (468.31,30.12) .. controls (471.68,30.12) and (474.41,32.85) .. (474.41,36.22) -- (474.41,54.52) .. controls (474.41,57.88) and (471.68,60.61) .. (468.31,60.61) -- (297.31,60.61) .. controls (293.94,60.61) and (291.21,57.88) .. (291.21,54.52) -- cycle ;
\draw  [line width=0.75]  (505.07,41.48) .. controls (505.07,39.62) and (506.58,38.11) .. (508.44,38.11) -- (734.04,38.11) .. controls (735.9,38.11) and (737.41,39.62) .. (737.41,41.48) -- (737.41,51.58) .. controls (737.41,53.44) and (735.9,54.95) .. (734.04,54.95) -- (508.44,54.95) .. controls (506.58,54.95) and (505.07,53.44) .. (505.07,51.58) -- cycle ;
\draw  [color={rgb, 255:red, 74; green, 74; blue, 74 }  ,draw opacity=1 ][fill={rgb, 255:red, 197; green, 197; blue, 197 }  ,fill opacity=1 ][line width=0.75]  (600.84,88.43) -- (609.14,83.74) -- (609.14,86.09) -- (625.75,86.09) -- (625.75,83.74) -- (634.05,88.43) -- (625.75,93.13) -- (625.75,90.78) -- (609.14,90.78) -- (609.14,93.13) -- cycle ;
\draw  [color={rgb, 255:red, 74; green, 74; blue, 74 }  ,draw opacity=1 ][fill={rgb, 255:red, 197; green, 197; blue, 197 }  ,fill opacity=1 ][line width=0.75]  (256.23,73.9) -- (324.82,73.9) -- (324.82,67.68) -- (334.23,80.12) -- (324.82,92.56) -- (324.82,86.34) -- (256.23,86.34) -- cycle ;
\draw  [color={rgb, 255:red, 74; green, 74; blue, 74 }  ,draw opacity=1 ][fill={rgb, 255:red, 197; green, 197; blue, 197 }  ,fill opacity=1 ][line width=0.75]  (330.23,111.37) -- (263.4,111.37) -- (263.4,114.35) -- (254.23,108.4) -- (263.4,102.45) -- (263.4,105.42) -- (330.23,105.42) -- cycle ;
\draw [color={rgb, 255:red, 74; green, 74; blue, 74 }  ,draw opacity=1 ][line width=2.25]    (299.62,99.91) -- (288.07,115.88) ;
\draw [color={rgb, 255:red, 74; green, 74; blue, 74 }  ,draw opacity=1 ][line width=2.25]    (301.42,115.55) -- (288.07,101.25) ;

\draw (78,39) node [anchor=north west][inner sep=0.75pt]  [font=\small] [align=left] {\footnotesize \textbf{In a Monoidal Category}};
\draw (280,34) node [anchor=north west][inner sep=0.75pt]  [font=\scriptsize] [align=left] {\begin{minipage}[lt]{150pt}\setlength\topsep{0pt}
\begin{center}
{\scriptsize \textbf{ In a Rigid, Abelian Monoidal }}\\[-.2pc]
{\scriptsize \textbf{Category w/ Simple Unit}}
\end{center}
\end{minipage}};
\draw (507,41) node [anchor=north west][inner sep=0.75pt]  [font=\small] [align=left] {\begin{minipage}[lt]{170pt}\setlength\topsep{0pt}
\begin{center}
{\footnotesize \textbf{In a Pivotal Multifusion Category}}
\end{center}
\end{minipage}};
\draw (28,70) node [anchor=north west][inner sep=0.75pt]   [align=left] {\begin{minipage}[lt]{80pt}\setlength\topsep{0pt}
\begin{center}
{\scriptsize Monadic}\\[-.3pc]{\scriptsize Division Algebra}\\[-.3pc]{\scriptsize Property}
\end{center}
\end{minipage}};
\draw (157,70) node [anchor=north west][inner sep=0.75pt]   [align=left] {\begin{minipage}[lt]{80pt}\setlength\topsep{0pt}
\begin{center}
{\scriptsize Essential}\\[-.3pc]{\scriptsize Division Algebra}\\[-.3pc]{\scriptsize Property}
\end{center}
\end{minipage}};
\draw (333,72) node [anchor=north west][inner sep=0.75pt]   [align=left] {\begin{minipage}[lt]{80pt}\setlength\topsep{0pt}
\begin{center}
{\scriptsize Simplistic}\\[-.3pc]{\scriptsize Division Algebra}\\[-.3pc]{\scriptsize Property}
\end{center}
\end{minipage}};
\draw (491,69) node [anchor=north west][inner sep=0.75pt]   [align=left] {\begin{minipage}[lt]{80pt}\setlength\topsep{0pt}
\begin{center}
{\scriptsize Left version of}\\[-.3pc]{\scriptsize Division Algebra}\\[-.3pc]{\scriptsize Property}
\end{center}
\end{minipage}};
\draw (635,69) node [anchor=north west][inner sep=0.75pt]   [align=left] {\begin{minipage}[lt]{80pt}\setlength\topsep{0pt}
\begin{center}
{\scriptsize Right version of}\\[-.3pc]{\scriptsize Division Algebra}\\[-.3pc]{\scriptsize Property}
\end{center}
\end{minipage}};
\draw (253,75) node [anchor=north west][inner sep=0.75pt]  [font=\tiny] [align=left] {\begin{minipage}[lt]{60pt}\setlength\topsep{0pt}
\begin{center}
$\displaystyle \exists \hspace{-.02in}$  simple module
\end{center}
\end{minipage}};
\end{tikzpicture}
}

    \vspace{-.1in}
    
    \caption{Summary of connections between  division algebra properties.}
    \label{fig:main}
\end{figure}

\vspace{-.09in}

\subsection*{Organization of article} In $\S$\ref{sec:background}, we provide background material on monoidal categories and algebraic objects within. In $\S$\ref{sec:defDA}, we introduce simplistic and essential division algebras  in the abelian monoidal setting,  explore examples and consequences in  rigid, multifusion, and pivotal multifusion categories, and establish Theorem~\ref{THM}(ii,iii). We then move to the monadic setting in~$\S$\ref{sec:monads}, where we provide background material, define monadic division algebras, and establish Theorem~\ref{THM}(i). Several questions and conjectures are discussed in $\S$\ref{sec:directions}.


\section{Background material}\label{sec:background}
Here, we recall monoidal categories and algebraic objects within. Monoidal categories are discussed in $\S$\ref{BG:moncat}, rigidity and other properties are discussed in $\S\S$\ref{subsec:rigpiv}, \ref{BG:types}. We discuss invertible objects in monoidal categories in $\S$\ref{ss:inv}, construct algebraic objects in monoidal categories in $\S$\ref{BG:algobjs}, and introduce module categories over a monoidal category in $\S$\ref{BG:modcats}. Morita equivalence in the monoidal context is reviewed in $\S$\ref{ss:Morita}. We refer the reader to \cite{EGNO, Walton} for more details.

\smallskip

Sometimes, we impose that categories are  abelian. In this case, the zero object is denoted by $\mathsf{0}$,  the biproduct is denoted by $\sqcup$, and  an object is {\it simple} if its only subobjects are itself and $\mathsf{0}$.

\subsection{Monoidal categories}\label{BG:moncat} 

A \textit{monoidal category} consists of a category $\cC$ equipped with a bifunctor $\otimes:\cC\times\cC\to\cC$, a unit object $\one\in\cC$, and associator and unitality natural isomorphisms satisfying the pentagon and triangle axioms.
By MacLane's strictness theorem, we assume throughout that all  monoidal categories are \textit{strict}, meaning that the associator and unitality natural isomorphisms are all equal to the identity natural isomorphism on $\cC$.

\smallskip

\begin{center}
For all that follows, let $\cC:=(\cC,\otimes,\one)$ denote a monoidal category. 
\end{center}

\subsection{Rigidity and pivotality}\label{subsec:rigpiv}

We say that  $\cC$ is {\it rigid} if each $X \in \cC$ has a {\it left dual} $X^* \in \cC$ with (co)evaluation maps $\text{ev}_X^L:  X^* \otimes X \to \one$ and $\text{coev}_X^L:  \one \to  X \otimes X^*$, and a {\it right dual} ${}^*X \in \cC$ with (co)evaluation maps $\text{ev}_X^R:  X \otimes {}^*X \to \one$, $\text{coev}_X^R: \one \to  {}^*X \otimes X$,  satisfying  coherence conditions. 
Here, $\rdual{(X^*)}\cong X\cong (\rdual X)^*$ in $\cC$.

\smallskip

For a morphism $f: X \to Y$ in a rigid category $\cC$, the left dual morphism $f^*:Y^*\to X^*$ and  right dual morphism $\rdual f:\rdual Y\to\rdual X$ exist. This gives that $(-)^*$ and $\rdual{(-)}$ are contravariant (strong monoidal) autoequivalences of $\cC$, called the \textit{left} and \textit{right duality functors}, respectively. 

\smallskip

A rigid category is called \textit{pivotal} if there exists a monoidal natural isomorphism $j:\Id_\cC\natisom(-)^{**}$, which exists if and only if there exists a monoidal natural isomorphism $\hat\jmath:(-)^*\natisom\rdual{(-)}$.

\subsection{Types of monoidal categories} \label{BG:types}

We say that $\cC$ is {\it abelian monoidal} if its underlying category is abelian  and the endofunctors $(X \otimes -)$ and $(- \otimes X)$ of $\cC$ are additive, for each $X \in \cC$. 

\smallskip

Recall that $\Bbbk$ denotes an algebraically closed field of characteristic~0. 
A category is \textit{$\Bbbk$-linear} if all Hom-sets are  $\Bbbk$-vector spaces, where composition distributes over addition and scalar multiplication. We say that $\cC$ is ($\Bbbk$-){\it linear monoidal} if the underlying category is $\Bbbk$-linear, and the endofunctors $(X \otimes -)$ and  $(- \otimes X)$ of $\cC$ are linear, for each $X \in \cC$. 

\smallskip

A $\Bbbk$-linear abelian category $\cC$ is \textit{locally finite} if every object has finite length and each Hom-space is finite dimensional. It is \textit{finite} if it is locally finite, has enough projectives, and has only finitely many isoclasses of simple objects. 
We say that $\cC$ is \textit{fusion} if it is abelian, $\Bbbk$-linear monoidal, finite, rigid, semisimple, and satisfies $\End_\cC(\one)\cong \Bbbk$. When  the last condition is omitted,  $\cC$ is  \textit{multifusion}.

\subsection{Invertible objects in monoidal categories} \label{ss:inv} There are a few notions of invertible objects in monoidal categories. Here, $X \in \cC$ is {\it left invertible} if there exists  $X^L \in \cC$ such that $X^L \otimes X \cong \one$, and is {\it right invertible} if there exists  $X^R \in \cC$ such that $X \otimes X^R \cong \one$. We also say that $X$ is {\it invertible} if it is both left and right invertible; here, $X^L \cong X^R$.
The result below is straightforward to verify.

\pagebreak 

\begin{lemma} We have the following statements about (one-sided) invertible objects in $\cC$.
 \begin{enumerate}\renewcommand{\labelenumi}{\textnormal{(\roman{enumi})}} \label{lem:invertible}
\item An object $X \in \cC$ is left invertible if and only if  $(- \otimes X): \cC \to \cC$ is essentially~surjective. 
\smallskip
\item An object $X \in \cC$ is right invertible if and only if $(X \otimes -): \cC \to \cC$ is essentially~surjective. 
\smallskip
\item If $X \in \cC$ is invertible,  then the functors $(- \otimes X), (X \otimes -): \cC \to \cC$ are equivalences. \qed
\end{enumerate}
\end{lemma}

In a rigid category $\cC$, a stronger notion of invertibility for $X \in \cC$  is to require that it is invertible  via (co)evaluation morphisms \cite[$\S$2.11]{EGNO}. 
Next, we connect invertibility here with simplicity. 

\begin{lemma} \label{lem:invsimple} Take an abelian monoidal category $\cC$ with simple $\one$, and where all objects have finite length. If $X \in \cC$ is left or right invertible, then $X$ is simple. 
\end{lemma}

\begin{proof} This follows since  $\text{length}(X \otimes Y) \geq \text{length}(X) \, \text{length}(Y)$ when $X,Y \in \cC$ have finite length \cite[Exercise~4.3.11(1,2)]{EGNO}, and objects of length 1 are precisely the simple objects of~$\cC$.
\end{proof}

\subsection{Algebraic structures in monoidal categories}\label{BG:algobjs} 

An \textit{algebra in $\cC$} is a tuple $(A,m_A,u_A)$ consisting of an object $A\in\cC$ and morphisms $m_A:A\otimes A\to A$ and $u_A:\one\to A$ in $\cC$ which satisfy associativity and unitality constraints. These objects form a category, $\mathsf{Alg}(\cC)$, where morphisms are those morphisms in $\cC$ that respect the algebra structures. 

\medskip 

For example, $\one \in \mathsf{Alg}(\cC)$, with $m_\one : \one \otimes \one \equivto \one$ (unitor isomorphism), and $u_\one = \id_\one$. In an abelian monoidal category, $\mathsf{0} \in \mathsf{Alg}(\cC)$ with  $m_\mathsf{0} : \mathsf{0} \otimes \mathsf{0} \to \mathsf{0}$  and $u_\mathsf{0} = \one \to \mathsf{0}$ coming from $\mathsf{0}$ being terminal.

\medskip

Fix an algebra $(A,m_A,u_A)$ in $\cC$. A \textit{right $A$-module in $\cC$} is a pair $(M,\actR)$, where $M\in\cC$ and $\actR:M\otimes A\to M$ in $\cC$ satisfying associativity and unitality axioms. These structures form a category $\Mod\text{-}A(\cC)$, where morphisms are those morphisms in $\cC$ that respect the right module structures. \textit{Left $A$-modules $(N,\actL)$ in $\cC$} and the category $A\text{-}\Mod(\cC)$ are defined likewise.

\medskip

If $(B,m_B,u_B)$ is another algebra in $\cC$, an \textit{$(A,B)$-bimodule in $\cC$} is a tuple $(Q,\actL,\actR)$ such that $(Q,\actL)\in A\text{-}\Mod(\cC)$, $(Q,\actR)\in\Mod\text{-}B(\cC)$, satisfying a middle associativity axiom. With morphisms that are simultaneously left and right module morphisms, we obtain the category $(A,B)\text{-}\Bimod(\cC)$.

\medskip

For example, in an abelian monoidal category $\cC$, we have that $\mathsf{0} \in A\text{-}\Mod(\cC)$, with  $\actL : A \otimes \mathsf{0} \to \mathsf{0}$ coming from~$\mathsf{0}$ being terminal.

\medskip 

 A non-zero right module $M\in\Mod\text{-}A(\cC)$ is called  {\it simple} if it is a simple object in the category $\Mod\text{-}A(\cC)$; a similar notion holds for right modules and bimodules in~$\cC$.

\medskip 

The \textit{regular right} (resp., {\it left}) {\it $A$-module in $\cC$} is $(A,m_A)$ in $\Mod\text{-}A(\cC)$ (resp.,  in $A\text{-}\Mod(\cC)$), and \textit{the regular $(A,A)$-bimodule in $\cC$} is $(A,m_A,m_A)$ in $(A,A)\text{-}\Bimod(\cC)$, which is denoted by $A_{\text{reg}}$ or $A$.

\medskip

Moreover, a right $A$-module $(M,\actR)$ in $\cC$ is  said to be \textit{free} if there is an object $X\in\cC$ so that $(M,\actR)\cong (X\otimes A,\id_X\otimes m_A)$ in $\Mod\text{-}A(\cC)$. Similarly, a left $A$-module $(N,\actL)$ in $\cC$ is said to be \textit{free} if there is an object $Y\in\cC$ so that $(N,\actL)\cong (A\otimes Y,m_A\otimes \id_Y)$ in $A\text{-}\Mod(\cC)$. For instance, the regular left and right $A$-modules are the free modules over $A$ on the object $\one\in\cC$.


\medskip 

Now let $(M,\actR)\in\Mod\text{-}A(\cC)$. If $\rdual M$ exists in $\cC$, then $\rdual M$ is a left $A$-module in $\cC$. Similarly, given $(N,\actL)\in A\text{-}\Mod(\cC)$, if $N^*$ exists in $\cC$, then $N^*$ is a right $A$-module in $\cC$.
When $\cC$ is rigid, restricting the duality functors from $\S$\ref{subsec:rigpiv} to categories of modules, we get the equivalences $(-)^*:A\text{-}\Mod(\cC)\xrightarrow{\sim} \Mod\text{-}A(\cC)$ and $\rdual(-):\Mod\text{-}A(\cC)\xrightarrow{\sim} A\text{-}\Mod(\cC)$, for any $A\in\Alg(\cC)$.

\medskip
 
 Given $(M,\actR)\in\Mod\text{-}A(\cC)$ and $(N,\actL)\in A\text{-}\Mod(\cC)$, the \textit{tensor product of $M$ and $N$ over $A$} is the coequalizer of the morphisms $\id_M\otimes \actL$ and $\actR\otimes\id_N$, denoted by $M\otimes_A N$, if it exists in~$\cC$. In any case, for any $Q\in (A,B)\text{-}\Bimod(\cC)$ and $P\in (B,C)\text{-}\Bimod(\cC)$, we get that $Q\otimes_B P$ is in $(A,C)\text{-}\Bimod(\cC)$. Here, $(A,A)\text{-}\Bimod(\cC)$ is monoidal with $\otimes:=\otimes_A$ and $\one:=A_{\text{reg}}$.

\medskip

A \textit{left ideal} of an algebra $A$ in $\cC$ is a subobject of  $A_{\text{reg}} \in A\text{-}\Mod(\cC)$. In other words, it is an object $(I,\lambda)$ with a mono $\iota_I^A:(I,\lambda)\hookrightarrow (A,m_A)$ in $A\text{-}\Mod(\cC)$. Similarly, a \textit{right ideal of $A$} is a subobject of $A_{\text{reg}} \in \Mod\text{-}A(\cC)$, and a \textit{\textnormal{(}two-sided\textnormal{)} ideal of $A$} is a subobject of $A_{\text{reg}} \in (A,A)\text{-}\Bimod(\cC)$.


\medskip

In an abelian monoidal category $\cC$, a non-zero algebra $A$ is \textit{simple} if its only ideals are itself and zero (i.e., $A_{\text{reg}}$ is a simple object in $(A,A)\text{-}\Bimod(\cC)$). 
Next, consider the  preliminary result below.

\begin{lemma} \label{lem:zeromod} In an abelian monoidal category $\cC$,  an algebra $A$ is $\cC$ is the zero algebra $\mathsf{0}$ if and only if $A$-$\mathsf{Mod}(\cC)$ and $\mathsf{Mod}$-$A(\cC)$ each have one object, namely the zero module, up to isomorphism.
\end{lemma}

\begin{proof}
If $A = \mathsf{0}$ and $(M, \actL) \in \mathsf{0}$-$\mathsf{Mod}(\cC)$, then $\id_M = \actL  \circ (u_{\mathsf{0}} \otimes \id_M)$ is a zero morphism. Hence, $M$ is both initial and terminal, and $M \cong \mathsf{0}$. Similarly, all right modules over $A=\mathsf{0}$ are also zero. 
On the other hand, if $A$ is non-zero, then $A$-$\mathsf{Mod}(\cC)$ (or $\mathsf{Mod}$-$A(\cC)$) contains the zero module and the regular module, and these are not isomorphic.
\end{proof}

\subsection{Module categories} \label{BG:modcats}


A \textit{left $\cC$-module category} consists of a category $\cM$, a left action bifunctor $\actL:\cC\times\cM\to\cM$, and associativity and unitality natural isomorphisms which satisfy the pentagon and triangle axioms. \textit{Right $\cC$-module categories} $(\cN, \actR)$ are defined likewise.

\medskip

The \textit{regular left (resp. right) $\cC$-module category} is given by $\cC$, with action bifunctor $\actL:=\otimes$ (resp. $\actR:=\otimes$). We also have that for any algebra $A\in\Alg(\cC)$, the category $\Mod\text{-}A(\cC)$ is a left $\cC$-module category and $A\text{-}\Mod(\cC)$ is a right $\cC$-module category, again with action bifunctors given by $\otimes$.




\medskip

A left  $\cC$-module category $\cM:=(\cM,\actL)$  is  \textit{closed} if, for each $M\in\cM$ (resp,. $N\in\cN$), the functor $(-\actL M):\cC\to\cM$  has a right adjoint: $\underline{\Hom}_\cM(M,-):\cM\to\cC$. We call $\underline{\Hom}_{\cM}(M,N)$ the \textit{internal Hom of $M$ and $N$}. Also, $\underline{\End}_{\cM}(M):= \underline{\Hom}_{\cM}(M,M)$ is the \textit{internal End of $M$}. Similar notions hold for right $\cC$-module categories.


\medskip

For any $M\in(\cM,\actL)$ and any $N\in (\cN, \actR)$, the objects $\underline{\End}_\cM(M)$ and $\underline{\End}_\cN(N)$ are algebras in~$\cC$. Given $M'\in\cM$ and $N'\in\cN$, we obtain that $\underline{\Hom}_\cM(M,M')$ is a right $\underline{\End}_\cM(M)$-module in $\cC$. 
Similarly, $\underline{\Hom}_\cN(N,N')$ is a 
left $\underline{\End}_{\cN}(N')$-module in $\cC$. 
From this, we obtain the functors $\underline{\Hom}_\cM(M,-):\cM\to \Mod\text{-}\underline{\End}_\cM(M)(\cC)$ and $\underline{\Hom}_\cN(-,N'):\cN\to\underline{\End}_\cN(N')\text{-}\Mod(\cC)$.

\medskip

As an  example, if the category $\cC$ is rigid, the regular left $\cC$-module category is closed, with $\underline{\Hom}_{\cC}(X,Y)\cong Y\otimes X^*$. The algebra and module structures on these internal Homs and Ends are then given by appropriate evaluation and coevaluation maps. 

\medskip

Given an algebra $A\in\cC$, we have that the left $\cC$-module category $\Mod\text{-}A(\cC)$  and the right $\cC$-module category $A\text{-}\Mod(\cC)$ are both closed, with $\underline{\Hom}_{\Mod\text{-}A(\cC)}(M,M')\cong (M\otimes_A\rdual(M'))^*$ and $\underline{\Hom}_{A\text{-}\Mod(\cC)}(N,N')\cong \rdual((N')^*\otimes_A N)$ in $\cC$.   The next result is also useful.

\begin{lemma} \label{lem:A-End(A)}
For $\cC$ rigid with $A \in \mathsf{Alg}(\cC)$, we get $A \cong \underline{\End}_{A\text{-}\Mod(\cC)}(A) \cong \underline{\End}_{\Mod\text{-}A(\cC)}(A)$ in $\mathsf{Alg}(\cC)$.
\end{lemma}

\begin{proof}
Take the projection $\pi: A^* \otimes A \to  A^* \otimes_A A$ from the coequalizer property. Next, note that $m_{\underline{\End}_{A\text{-}\Mod(\cC)}(A)} = {}^* \mu$ where $\mu:= \id_{A^*} \otimes_A \text{coev}_A^L \, \otimes_A \id_A$, and $u_{\underline{\End}_{A\text{-}\Mod(\cC)}(A)} = {}^* \eta$ such that $\text{ev}_A^L = \eta \, \pi$. Moreover, the isomorphism $A^* \otimes_A A \cong A^*$ in $\cC$ is given by mutually inverse morphisms  $\phi: A^* \otimes_A A \to A^*$ and $\psi: A^* \to A^* \otimes_A A$ in $\cC$, where 
\[
\phi \, \pi = \triangleleft_{A^*} = (\text{ev}_A^L \otimes \id_{A^*})(\id_{A^*} \otimes m_A \otimes \id_{A^*})(\id_{A^*} \otimes \id_A \otimes \text{coev}_A^L) \;\; \;\text{and} \;\;\; \psi = \pi \, (\id_{A^*} \otimes u_A).
\]
Now one can check that $(m_A)^* = (\phi \otimes \phi) \, \mu \, \psi$ and $(u_A)^* = \eta \, \psi$. Thus,  ${}^*\psi$ yields the first algebra isomorphism in $\cC$. Similarly, $A \cong \underline{\End}_{\Mod\text{-}A(\cC)}(A)$ as algebras in $\cC$.
\end{proof}

\subsection{Morita equivalence of algebras.} \label{ss:Morita} 
Given algebras $A$ and $B$ in $\cC$, we say that $A$ and $B$ are \textit{Morita equivalent} in $\cC$ if either of the equivalent conditions holds: $A\text{-}\Mod(\cC) \simeq B\text{-}\Mod(\cC)$  as right $\cC$-module categories, or  $\Mod\text{-}A(\cC) \simeq \Mod\text{-}B(\cC)$  as left $\cC$-module categories.

\smallskip

For instance, let $\cC$ be abelian rigid monoidal  with simple unit. Then, for any non-zero $X\in\cC$, the internal End of $X$,  given by $X\otimes X^*$, is Morita equivalent to $\one$ in $\cC$ via one of the functors below:
 \begin{equation} \label{eq:XX*}
 (-\otimes X^*):\cC\xrightarrow{\sim}\Mod\text{-}(X\otimes X^*)(\cC);\qquad (X\otimes -):\cC\xrightarrow{\sim}(X\otimes X^*)\text{-}\Mod(\cC).
 \end{equation}
Their respective quasi-inverses are given by: 
\[
(- \otimes_{X \otimes X^*} X): \Mod\text{-}(X\otimes X^*)(\cC) \to \cC; \qquad (X^* \otimes_{X \otimes X^*} -): (X\otimes X^*)\text{-}\Mod(\cC) \to \cC.
\]
See, e.g., \cite[Example 4.58]{Walton}. Similarly, the algebra ${}^*X \otimes X$ is also Morita equivalent to $\one$ in $\cC$.

\smallskip

The categories $A\text{-}\Mod(\cC)$ and $\Mod\text{-}A(\cC)$ are the prototypical examples of $\cC$-module categories, and the following theorem from \cite{Ostrik} addresses when any given $\cC$-module category is of this form.

\begin{theorem}[Ostrik's Theorem] \label{Thm:Ostrik} Let $\cC$ be a multifusion category, with $\cM$ a non-zero, indecomposable left $\cC$-module category and $\cN$ a non-zero, indecomposable, right $\cC$-module category. Then, for any non-zero $M\in\cM$ and any non-zero $N' \in\cN$, we have that
\[
\cM\simeq \Mod\text{-}\underline{\End}_\cM(M)(\cC)\qquad\text{and}\qquad \cN\simeq \underline{\End}_{\cN}(N')\text{-}\Mod(\cC),
\]

\noindent as left and right $\cC$-module categories, respectively, via $\underline{\textnormal{Hom}}_\cM(M,-):\cM\xrightarrow{\sim}\Mod\text{-}\underline{\End}_\cM(M)(\cC)$ and $\underline{\Hom}_\cN(-,N'):\cN\xrightarrow{\sim}\underline{\End}_{\cN}(N')\text{-}\Mod(\cC)$.
\qed
\end{theorem}


\section{Module-theoretic division algebras}\label{sec:defDA}

In this part, we adapt Definition~\ref{def:DA}(ii,iii) to the abelian monoidal setting. In $\S$\ref{sec:genMon}, we introduce  module-theoretic division algebras in abelian monoidal categories. We then explore these structures in rigid, multifusion, and pivotal multifusion  categories in $\S\S$\ref{sec:rigid}, \ref{sec:multifusion}, \ref{sec:pivotal}, respectively.

\subsection{In abelian monoidal categories}\label{sec:genMon}

Let $\cC$ denote an abelian monoidal category.

\begin{definition}\label{def:SimpDA}
     A non-zero algebra $A\in\Alg(\cC)$ is a \textit{left} (resp., {\it right}) {\it simplistic division algebra} in~$\cC$ if the regular module $A_{\text{reg}}$ in $A\text{-}\Mod(\cC)$ (resp., in $\Mod\text{-}A(\cC)$) is simple, and we say that $A$ is a {\it simplistic division algebra} in~$\cC$ if both conditions hold.

The full subcategories of $\Alg(\cC)$ consisting of such objects are denoted by $\ell.\mathsf{SimpDivAlg}(\cC)$, by  $r.\mathsf{SimpDivAlg}(\cC)$, and by $\mathsf{SimpDivAlg}(\cC)$, respectively.
\end{definition}

\begin{definition}\label{def:EssDA}
     A non-zero algebra $A\in\Alg(\cC)$ is a \textit{left} (resp., \textit{right}) \textit{essential division algebra in~$\cC$} if the functor $(A\otimes -):\cC\to A\text{-}\Mod(\cC)$ (resp.,  $(-\otimes A):\cC\to\Mod\text{-}A(\cC)$) is essentially surjective. 
     We say that $A$ is an \textit{essential division algebra in~$\cC$} if both conditions hold.

 The full subcategories of $\Alg(\cC)$ consisting of such objects are denoted by $\ell.\mathsf{EssDivAlg}(\cC)$, by $r.\mathsf{EssDivAlg}(\cC)$, and by $\mathsf{EssDivAlg}(\cC)$, respectively.
\end{definition}

Note that Definition~\ref{def:SimpDA} was used in previous works involving division algebras in abelian monoidal categories \cite{GrossmanSnyder, Grossman, KongZheng}, as this recovers Definition~\ref{def:DA}(iii) when $\cC$ is the monoidal category of $\Bbbk$-vector spaces, $(\mathsf{Vec}, \otimes_\Bbbk, \Bbbk)$. On the other hand, Definition~\ref{def:EssDA} recovers Definition~\ref{def:DA}(ii) when~$\cC$ is $\mathsf{Vec}$ since if the functor $(-\otimes A)$ is essentially surjective, then every right $A$-module in $\cC$ is isomorphic to one in the image of $(-\otimes A)$, hence free. Similarly, if $(A\otimes -)$ is essentially surjective, then every left $A$-module in $\cC$ is free.
Moreover, the hypothesis that $A$ is non-zero in the terminology above is needed; else, by Lemma~\ref{lem:zeromod}, the conditions in Definitions~\ref{def:SimpDA} and \ref{def:EssDA} hold vacuously.

\begin{example} \label{ex:one}
If $\one \neq \mathsf{0}$, then $\one$ is  a  simplistic division algebra precisely when $\one$ is a simple object in $\cC$. But, $\one$ is always an essential division algebra since, by unitality, the functors $(\one\otimes-):\cC\to\cC$ and $(-\otimes\one):\cC\to\cC$ are  essentially surjective.
\end{example}

With the quick example above, we see that these two types of division algebras differ in the general abelian monoidal setting beyond $\mathsf{Vec}$. 

\subsection{In rigid, abelian monoidal categories with simple unit}\label{sec:rigid}
In this part, assume that $\cC$  is a rigid, abelian monoidal category with simple $\one$.
We will show that essential division algebras in~$\cC$ are simplistic division algebras in~$\cC$. We will also  present examples of  simplistic, non-essential division algebras in $\cC$, showing that these definitions remain distinct in this setting.

\begin{proposition}\label{Prop:ess to simp} Take  $A \in \mathsf{Alg}(\cC)$ where $A$ admits a simple left (resp., right) module in $\cC$. If $A$ is in $\ell./r.\mathsf{EssDivAlg}(\cC)$, then $A$ is in $\ell./r.\mathsf{SimpDivAlg}(\cC)$. 
\end{proposition}

\begin{proof}
We start with $A\in\Alg(\cC)$ that is a left essential division algebra admitting a simple left $A$-module $S$ in $\cC$. 
We aim to show that  $A_{\text{reg}} \in A\text{-}\Mod(\cC)$ is simple.
Now take a non-zero left ideal $\iota: I \rightarrow A$ in $\cC$; it suffices to show that the mono $\iota$ is an isomorphism in $\cC$. 

Since $A$ is a left essential division algebra, the module $S$ is free. So, there is an object $X \in\cC$ such that $S\cong A\otimes X\in A\text{-}\Mod(\cC)$. By exactness of the functor $(-\otimes X)$, a consequence of rigidity, we have that monos are preserved. Thus, we obtain a submodule $\iota \otimes \id_X: I \otimes X \rightarrow A\otimes X$. By simplicity of  $S\cong A\otimes X$, either $I\otimes X = \mathsf{0}$ or $\iota\otimes\id_X$ is an isomorphism.

Next, note that $(-\otimes X):\cC\to \Mod\text{-}(\rdual{X}\otimes X)(\cC)$ is an equivalence of categories via the right dual version of \eqref{eq:XX*}. Since $I\neq\mathsf{0}$, we conclude that $I\otimes X$, the image of $I$ under the equivalence $(-\otimes X)$, is also non-zero. Hence,  $\iota \otimes \id_X$ is an isomorphism in $\cC$, and hence in $(\rdual{X}\otimes X)\text{-}\Mod(\cC)$. But, equivalences of categories reflect isomorphisms, so $\iota$ must be an isomorphism, as desired.

 The right version argument is similar, using the equivalence $(X\otimes -):\cC\xrightarrow{\sim}(X\otimes X^*)\text{-}\Mod(\cC)$.
\end{proof}

\begin{remark} \label{rem:hypsimple}
The hypothesis on $A$ in Proposition~\ref{Prop:ess to simp} holds when the regular module in $A$-$\mathsf{Mod}(\cC)$ (resp., $\mathsf{Mod}$-$A(\cC)$) is left (resp., right) Artinian, is left (resp., right) Noetherian, or is semisimple.  
\end{remark}

Next, we construct simplistic, non-essential division algebras in $\cC$. To do this, we study when the internal End algebras of the regular left and right $\cC$-module categories are division algebras in $\cC$.

\begin{proposition}\label{prop:IntEndDA}
Take $\cC$ as above, and take an object $X$ in $\cC$.
    
    \begin{enumerate}\renewcommand{\labelenumi}{\textnormal{(\roman{enumi})}}
        \item $X\otimes X^*$ is a simplistic division algebra in $\cC$ if and only if $X$ is simple.
        
        \smallskip
             
        \item ${}^*X \otimes X$ is a simplistic division algebra in $\cC$ if and only if $X$ is simple.
       
        \smallskip

        \item $X\otimes X^*$ is an essential division algebra in $\cC$ if and only if $X$ is left invertible. 
                
        \smallskip
        
        \item ${}^*X\otimes X$ is an essential division algebra in $\cC$ if and only if $X$ is right invertible.
    \end{enumerate}
\end{proposition}

\begin{proof}
For (i), recall the equivalence 
$(-\otimes X^*):\cC\xrightarrow{\sim}\Mod\text{-}(X\otimes X^*)(\cC)$
 from \eqref{eq:XX*}.
Applying this to $X$, we obtain that $X$ is simple in $\cC$ if and only if $X\otimes X^*$ is simple in $\Mod\text{-}(X\otimes X^*)(\cC)$, if and only if $X\otimes X^*$ is a right simplistic division algebra. Again, using \eqref{eq:XX*}, we get an equivalence $X\otimes (-)^*:\cC\xrightarrow{\sim} (X\otimes X^*)\text{-}\Mod(\cC)$, and applying this to $X$, we obtain  that $X$ is simple in $\cC$  if and only if $X\otimes X^*$ is a left simplistic division algebra. The proof of part (ii) follows likewise.

    For (iii), by way of Lemma~\ref{lem:invertible}(i), we first show that $X\otimes X^*$ is a right essential division algebra if and only if $(-\otimes X):\cC\to\cC$ is essentially surjective. For the forward direction, assume that $X\otimes X^*$ is a right essential division algebra, and let $Z\in\cC$ be any object. Take the module $Z\otimes X^*$ in $\Mod\text{-}(X\otimes X^*)(\cC)$, and  by the assumption, there is an object $\tilde Z\in\cC$ such that $\tilde Z\otimes X\otimes X^*\cong Z\otimes X^*$ in $\Mod\text{-}(X\otimes X^*)(\cC)$. Again, $(-\otimes X^*):\cC\xrightarrow{\sim}\Mod\text{-}(X\otimes X^*)(\cC)$ is an equivalence, so apply its quasi-inverse to get that $\tilde Z\otimes X\cong Z$ in $\cC$. Thus, $Z$ is in the essential image of $(-\otimes X)$.

    Conversely, assume that $(-\otimes X):\cC\to\cC$ is essentially surjective, and let $M\in\Mod\text{-}(X\otimes X^*)(\cC)$ be any right module in $\cC$. Since $(-\otimes X^*)$  is essentially surjective onto $\Mod\text{-}(X\otimes X^*)(\cC)$, there exists an object $\tilde M\in\cC$ such that $\tilde M\otimes X^*\cong M$ in $\Mod\text{-}(X\otimes X^*)(\cC)$. Moreover, since $(-\otimes X)$ is essentially surjective onto $\cC$, there exists an object $\tilde X\in\cC$ such that $\tilde X\otimes X\cong \one$ in $\cC$. Then, $(\tilde M\otimes \tilde X)\otimes (X\otimes X^*)\cong \tilde M\otimes X^*\cong M$ in $\Mod\text{-}(X\otimes X^*)(\cC)$, so that $M$ is in the essential image of $(-\otimes (X\otimes X^*)):\cC\to\Mod\text{-}(X\otimes X^*)(\cC)$, completing the direction.

Likewise, $X\otimes X^*$ is a left essential division algebra if and only if $(-\otimes X):\cC\to\cC$ is essentially surjective, by the equivalence $(X\otimes -):\cC\xrightarrow{\sim}\Mod\text{-}(X\otimes X^*)(\cC)$
 from \eqref{eq:XX*}, and by duality functors. Now apply Lemma~\ref{lem:invertible}(i) to conclude part (iii). The proof of part (iv) follows similarly.
 \end{proof}

\begin{remark} \label{rem:XX*}
We can recover Proposition~\ref{Prop:ess to simp}  for the algebras $X \otimes X^*$ and ${}^*X \otimes X$ in $\cC$, in the finite length case. First, by applying \eqref{eq:XX*} to the simple object $\one \in \cC$, we get simple modules $X^* \in \mathsf{Mod}\text{-}(X \otimes X^*)(\cC)$ and $X \in (X \otimes X^*)\text{-}\mathsf{Mod}(\cC)$. So, the hypotheses of  Proposition~\ref{Prop:ess to simp}  hold for the algebra $X \otimes X^*$ in $\cC$. Now assume that $X \otimes X^*$ is an essential division algebra  in $\cC$. Then, $X$ is left invertible [Proposition~\ref{prop:IntEndDA}(iii)], so $X$ is simple [Lemma~\ref{lem:invsimple}], and hence $X \otimes X^*$ is a simplistic division algebra [Proposition~\ref{prop:IntEndDA}(i)].
Similar arguments work for the algebra ${}^*X \otimes X$ in~$\cC$.
\end{remark}

Next, we provide examples of simplistic, non-essential division algebras in certain fusion categories.
Indeed, fusion categories satisfy the hypotheses on $\cC$ here, including those in Remark~\ref{rem:XX*}.

\begin{example}\label{ex:FIB}
Take the Fibonacci fusion category, $\mathsf{Fib}$, which has simple objects $\one$ and $\tau$ satisfying the fusion rules:
$\one\otimes\one\cong\one$, and $\one\otimes\tau\cong\tau\cong\tau\otimes\one$, and $\tau\otimes\tau\cong \one\sqcup \tau.$
See e.g., \cite[$\S$3.9]{Walton} or \cite{BookerDavydov}.
We have that $\tau\otimes\tau^*\cong \one\sqcup\tau$. Since $\tau$ is simple in $\cC$, Proposition~\ref{prop:IntEndDA}(i) implies that $1\sqcup\tau$ is a simplistic division algebra in $\mathsf{Fib}$.

But, $(-\otimes \tau):\mathsf{Fib}\to\mathsf{Fib}$ is not essentially surjective. Indeed, since $\mathsf{Fib}$ is semisimple, each object in $\mathsf{Fib}$ is isomorphic to $\one^n\sqcup \tau^m$, for some $m,n\geq 0$. Then, the essential image of $(-\otimes\tau)$ has objects  $(\one^n\sqcup \tau^m)\otimes\tau\cong\one^m\sqcup\tau^{m+n}$, for $m,n\geq 0$. So, $\one$ is not in the essential image of $(-\otimes\tau)$, and by Lemma~\ref{lem:invertible}(i) with Proposition~\ref{prop:IntEndDA}(iii), $\one\sqcup\tau$ is not an essential division algebra in~$\mathsf{Fib}$. 
\end{example}

\begin{example} \label{ex:grouptype}
Take a finite non-abelian group $G$, and take its (fusion) category $\mathsf{FdRep}(G)$ of finite-dimensional representations over $\Bbbk$. Here, $\otimes:= \otimes_\Bbbk$ and $\one := \Bbbk$. Now $\mathsf{FdRep}(G)$ has a simple object~$Z$ with $\dim_\Bbbk(Z)>1$. So, $Z \otimes Z^*$ is a simplistic division algebra in $\mathsf{FdRep}(G)$ [Proposition~\ref{prop:IntEndDA}(i)]. 

But, $\dim_\Bbbk(\one) = 1$, and $\dim_\Bbbk(X \otimes Y) = \dim_\Bbbk(X)  \dim_\Bbbk(Y)$ for $X,Y \in \mathsf{FdRep}(G)$. So, $Z$ above is not left invertible, and $Z \otimes Z^*$ is not an essential division algebra in $\mathsf{FdRep}(G)$ [Proposition~\ref{prop:IntEndDA}(iii)]. 
\end{example}

\pagebreak 

\subsection{In multifusion categories} \label{sec:multifusion} 
Proposition~\ref{prop:IntEndDA} used the Morita equivalence of $\one$ and $X\otimes X^*$ for any non-zero $X\in\cC$. More generally,  the results below use a Morita equivalence from Ostrik's Theorem  [Theorem~\ref{Thm:Ostrik}].

\begin{proposition}\label{Prop:simpMF}
Let $\cC$ be multifusion with $A\in\Alg(\cC)$ whose categories of modules in $\cC$ satisfy the hypothesis of Ostrik's Theorem. 
    \begin{enumerate}\renewcommand{\labelenumi}{\textnormal{(\roman{enumi})}}
        \item $\underline{\End}_{\Mod\text{-}A(\cC)}(M)$ is a right simplistic division algebra if and only if $M$ is simple in $\Mod\text{-}A(\cC)$.
        
        \smallskip
        
        \item $\underline{\End}_{A\text{-}\Mod(\cC)}(N)$ is a left simplistic division algebra if and only if $N$ is simple in $A\text{-}\Mod(\cC)$.
    \end{enumerate}
\end{proposition}

\begin{proof}
    Applying Ostrik's Theorem first to $\Mod\text{-}A(\cC)$, we obtain that for any non-zero $M\in\Mod\text{-}A(\cC)$, the functor $\underline{\Hom}_{\Mod\text{-}A(\cC)}(M,-):\Mod\text{-}A(\cC)\to \Mod\text{-}\underline{\End}_{\Mod\text{-}A(\cC)}(M)(\cC)$ is an equivalence of categories. Applying this equivalence to $M\in\Mod\text{-}A(\cC)$, it follows that $\underline{\End}_{\Mod\text{-}A(\cC)}(M)$ is simple in $\Mod\text{-}\underline{\End}_{\Mod\text{-}A(\cC)}(M)(\cC)$ if and only if $M$ is simple in $\Mod\text{-}A(\cC)$, proving (i).
    A similar application of Ostrik's Theorem to the right $\cC$-module category $A\text{-}\Mod(\cC)$ gives (ii).
\end{proof}

\begin{remark}
Deriving simplistic division algebras from  internal End algebras of simple objects was considered in \cite[Theorem 2.5]{Grossman}, \cite[Theorem 2.8]{GrossmanSnyder}, and \cite[Lemma 3.6(4)]{KongZheng}, without considering the converse statement. 
\end{remark}

\begin{proposition}\label{Prop:essMF}
Let $\cC$ be multifusion with $A\in\Alg(\cC)$ whose categories of modules in $\cC$ satisfy the hypothesis of Ostrik's Theorem. 
\begin{enumerate}\renewcommand{\labelenumi}{\textnormal{(\roman{enumi})}}
\item For any $M\in\Mod\text{-}A(\cC)$, we have that $\underline{\End}_{\Mod\text{-}A(\cC)}(M)$ is a right essential division algebra if and only if $(-\otimes M):\cC\to\Mod\text{-}A(\cC)$ is essentially surjective.
        
\smallskip
        
\item For any $N\in A\text{-}\Mod(\cC)$, we have that $\underline{\End}_{A\text{-}\Mod(\cC)}(N)$ is a left essential division algebra if and only if $(N\otimes -):\cC\to A\text{-}\Mod(\cC)$ is essentially surjective.
\end{enumerate}
\end{proposition}

\begin{proof}
To prove (i), note that $X\otimes \underline{\Hom}_{\Mod\text{-}A(\cC)}(M,M')\cong \underline{\Hom}_{\Mod\text{-}A(\cC)}(M,X\otimes M')$, for any $X\in\cC$ and $M,M'\in\Mod\text{-}A(\cC)$; see \cite[Lemma 7.9.4]{EGNO}. Therefore, we get that 
\[
(-\otimes\underline{\End}_{\Mod\text{-}A(\cC)}(M)) \, \cong \, \underline{\Hom}_{\Mod\text{-}A(\cC)}(M,-\otimes M)
\]
as functors from $\cC$ to $\Mod\text{-}\underline{\End}_{\Mod\text{-}A(\cC)}(M)(\cC)$. Moreover, $\underline{\Hom}_{\Mod\text{-}A(\cC)}(M,-\otimes M)$ is the composition of $(-\otimes M):\cC\to\Mod\text{-}A(\cC)$ and $\underline{\Hom}_{\Mod\text{-}A(\cC)}(M,-):\Mod\text{-}A(\cC)\to \Mod\text{-}\underline{\End}_{\Mod\text{-}A(\cC)}(M)(\cC)$, with the second being  an equivalence of categories by Ostrik's Theorem. Hence 
\[
(-\otimes\underline{\End}_{\Mod\text{-}A(\cC)}(M)) \, \cong \, \underline{\Hom}_{\Mod\text{-}A(\cC)}(M,-)\circ (-\otimes M)
\]
is essentially surjective if and only if $(-\otimes M)$ is essentially surjective, and we are done.
    The proof of (ii) is analogous.
\end{proof}

\begin{remark}\label{rem:schur} Propositions~\ref{Prop:simpMF} and \ref{Prop:essMF} are analogues of Schur's Lemma, which states that for a simple module $M$ over any non-zero $\Bbbk$-algebra, the $\Bbbk$-algebra $\End(M)$ is a division algebra in $\mathsf{Vec}_\Bbbk$. 
\end{remark}

\begin{example}\label{rem:A=1}
Let $A \in \mathsf{Alg}(\cC)$ such that $\Mod\text{-}A(\cC)$ satisfies the hypothesis of Ostrik's Theorem. Also, let $M\in\Mod\text{-}A(\cC)$ be a left invertible module, i.e., there is some $N\in A\text{-}\Mod(\cC)$ satisfying $N\otimes M\cong A_{\text{reg}}\in(A,A)\text{-}\Bimod(\cC)$. Then, $M'\cong M'\otimes_A A_{\text{reg}}\cong M'\otimes_A N\otimes M$ in $\Mod\text{-}A(\cC)$, for any $M'\in\Mod\text{-}A(\cC)$. Hence, $(-\otimes M):\cC\to\Mod\text{-}A(\cC)$ is essentially surjective, and  Proposition~\ref{Prop:essMF} gives that $\underline{\End}_{\Mod\text{-}A(\cC)}(M)\cong (M\otimes_A\rdual M)^*$ is a right essential division algebra in~$\cC$.

When $A = \one$, we recover Proposition~\ref{prop:IntEndDA}(iii): namely, if an object $X\in\cC$ is left invertible, then $(X \otimes \rdual{X})^* \cong X\otimes X^*$ is a right essential division algebra in $\cC$. 
\end{example}

\subsection{In pivotal multifusion categories}\label{sec:pivotal} 
We now address whether the distinction between left and right division algebras is necessary, and we find that in a pivotal multifusion category the distinction is not needed. Note that such categories are abundant,  as it is conjectured that every fusion category must be pivotal \cite[Conjecture~2.8]{EtingofNikshychOstrik}.

\begin{lemma} \label{lemma:pivalg}
Let $\cC$ be a pivotal abelian monoidal category with $A\in\Alg(\cC)$. Then the algebras $\underline{\End}_{\Mod\text{-}A(\cC)}(M)$ and $ \underline{\End}_{A\text{-}\Mod(\cC)}(\rdual M)$ are isomorphic in $\cC$, for any $M\in\Mod\text{-}A(\cC)$.
\end{lemma}

\begin{proof} 

We have the following isomorphism:
\[\underline{\End}_{\Mod\text{-}A(\cC)}(M) \cong (M\otimes_A\rdual{M})^* \xrightarrow{\hat{\jmath}_{M\otimes_A\rdual M}}\rdual{(}M\otimes_A\rdual M)\cong \underline{\End}_{A\text{-}\Mod(\cC)}(\rdual{M}).\]
Thus, it is suffices to show that $\hat{\jmath}_{M\otimes_A\rdual M}$ is an algebra map. To do this, recall that the algebra structure maps of $(M\otimes_A\rdual{M})^*$ are given by $m = (\mu)^*$ and $u = (\eta)^*$, while the algebra structure maps of $\rdual(M\otimes_A\rdual M)$ are given by $m' = \rdual{\mu}$ and $u' = \rdual{\eta}$, where
$\mu:= \id_M \otimes_A \operatorname{coev}_M^R \otimes_A \, \id_{\rdual M}$. Then, $\eta$ is defined as the map from $M\otimes_A\rdual{M}$ to $\one$ satisfying $\text{ev}_M^R = \eta \, \pi$, where $\pi$ is the projection associated to $M \otimes_A \rdual{M} = \text{coeq}(\actR\otimes\id_{\rdual M}, \, \id_M\otimes\actL)$.
Using this structure on the internal Ends, and the fact that $\hat \jmath$ is a monoidal natural transformation, it is straightforward to verify that $\hat\jmath$ is an algebra isomorphism.
\end{proof}

\begin{proposition}\label{prop:LiffR}
    Let $\cC$ be a pivotal multifusion category with $A\in\Alg(\cC)$.
    \begin{enumerate}\renewcommand{\labelenumi}{\textnormal{(\roman{enumi})}}
        \item $A\in\ell.\mathsf{SimpDivAlg}(\cC)$ if and only if $A\in r.\mathsf{SimpDivAlg}(\cC)$.
        \item $A\in\ell.\mathsf{EssDivAlg}(\cC)$ if and only if $A\in r.\mathsf{EssDivAlg}(\cC)$.
    \end{enumerate}
\end{proposition}

\begin{proof}
Start with algebras $A$ and $B$ in $\cC$ that are a left simplistic division algebra and left essential division algebra, respectively. Using the equivalence of categories $(-)^*$ from left modules to right modules we obtain that $A^*$ is simple in $\Mod\text{-}A(\cC)$, and that $(-\otimes B^*):\cC\to\Mod\text{-}B(\cC)$ is essentially surjective.
Proposition~\ref{Prop:simpMF}(i) then gives that $\underline{\End}_{\Mod\text{-}A(\cC)}(A^*)$ is a right simplistic division algebra, and Proposition~\ref{Prop:essMF}(i) gives that $\underline{\End}_{\Mod\text{-}B(\cC)}(B^*)$ is a right essential division algebra. By Lemmas~\ref{lem:A-End(A)} and~\ref{lemma:pivalg}, we get that as algebras, 
\[
\underline{\End}_{\Mod\text{-}A(\cC)}(A^*) \; \cong \; \underline{\End}_{A\text{-}\Mod(\cC)}(\rdual(A^*)) \; \cong \; \underline{\End}_{A\text{-}\Mod(\cC)}(A) \; \cong  \; A
\]
and similarly, $\underline{\End}_{\Mod\text{-}B(\cC)}(B^*)\cong B$. Thus, $A$ is a right simplistic division algebra, and $B$ is a right essential division algebra, as desired.
The backwards direction is analogous.
\end{proof}

\section{Monad-theoretic division algebras}\label{sec:monads}

Previously, we were restricted to working in abelian monoidal and (multi)fusion categories to study simplistic division algebras. But essential division algebras can be defined in any monoidal category; we will see here that they can be examined via monads. Background material on monads is in $\S$\ref{BG:Monads}. Monadic division algebras are introduced in $\S$\ref{subsec:MonDA}, and connections to essential division algebras are discussed there. Finally, we provide  examples of monadic division algebras in $\S$\ref{subsec:MontoEssDA}.

\subsection{Background on monads}\label{BG:Monads} References on monads include \cite[Chapter 5]{Riehl} and \cite[\S\S4.3.2, 4.4.3]{Walton}. For strong monads, see \cite{Moggi} or \cite{McDermottUustalu}.

\medskip

Let $\cA$ be any category. A \textit{monad on $\cA$} is an algebra in the monoidal category $(\End(\cA), \circ, \Id_\cA)$. More explicitly, a monad is a tuple $(T,\mu,\eta)$ where $T:\cA\to\cA$ is an endofunctor, and $\mu:T\circ T\Rightarrow T$ and $\eta:\Id_\cA\Rightarrow T$ are natural transformations satisfying associativity and unitality axioms.

\medskip

For example, given an adjunction, $(F: \cA \to \cB) \dashv (G: \cB \to \cA)$ with unit $\eta$ and counit $\varep$, we get that $(GF,G\varep F,\eta)$ is a monad on~$\cA$.

\medskip

Next, for a monad $(T,\mu,\eta)$ on $\cA$, the \textit{Eilenberg-Moore category $\cA^T$ of $T$} is the category with objects $(Y,\xi_Y)$ where $Y\in\cA$ and $\xi_Y:T(Y) \to Y \in \cA$. A morphism $f:(Y,\xi_Y)\to (Z,\xi_Z)$ is a morphism $f:Y\to Z\in\cA$, satisfying $f \circ \xi_Y = \xi_Z \circ T(f)$. 
This construction produces an adjunction 
\[(\operatorname{Free}^T :\cA\to\cA^T) \; \dashv  \; (\operatorname{Forg}^T:\cA^T\to\cA),\]
where $\operatorname{Free}^T(Y):= (T(Y),\mu_Y)$,  $\operatorname{Free}^T(f):= T(f)$, and $\operatorname{Forg}^T(Y, \xi_Y):= Y$, $\operatorname{Forg}^T(f):= f$.
The monad associated to this adjunction coincides with the original monad $T$. 

\medskip

Alternatively, the \textit{Kleisli category $\cA_T$ of $T$} is the category whose objects are the objects of~$\cA$, where  $\Hom_{\cA_T}(X,Y) = \Hom_{\cA}(X,T(Y))$. The composition $f\in\Hom_{\cA_T}(X,Y)$ and $g\in\Hom_{\cA_T}(Y,Z)$ is  given by $g\circ_T f := \mu_Z\circ T(g)\circ f$, which is in  $\Hom_{\cA}(X,T(Z))$. Again, this produces an adjunction:
\[(
F_T :\cA\to\cA_T) \; \dashv  \; (U_T:\cA_T\to\cA),\]
where $F_T(Y):= Y$,  $F_T(f:Y\to Z):= \eta_Z \circ f$, and $U_T(Y):= T(Y)$, $U_T(f:Y\to Z):= \mu_Z \circ T(f)$.
Again, the monad associated to this adjunction coincides with the original monad $T$. 

\medskip
 
 The category $\cA_T$ can be identified with the essential image of $\operatorname{Free}^T$ in $\cA^T$, via the embedding $K:\cA_T \to \cA^T$ given by $K(Y) = T(Y)$ and $K(f:Y \to Z) = \mu_Z \circ T(f)$. Hence, the objects of the Kleisli category of $T$ are considered as the free objects of the Eilenberg-Moore category of $T$.

\medskip

The Eilenberg-Moore and Kleilsi categories are, respectively, the terminal and initial solutions to the problem of finding an adjunction which gives rise to a certain monad. Namely, given a monad $T$ on $\cA$, consider the category $\mathsf{Adj}_T$, whose objects are adjunctions $(F:\cA\to\cB)\dashv (G:\cB\to\cA)$ which induce the monad $T$. Morphisms in this category are defined as follows. Given adjunctions $\mA_1:=(F_1:\cA\to\cB_1)\dashv(G_1:\cB_1\to\cA)$ and $\mA_2:=(F_2:\cA\to\cB_2)\dashv(G_2:\cB_2\to\cA)$, a morphism $K:\mA_1\to\mA_2$ is a functor $K:\cB_1\to\cB_2$ which satisfies $K\circ F_1 = F_2$ and $G_2\circ K = G_1$. 

\medskip

The unique functor $K_T:\cA_T\to\cA^T$ satisfying $K_T\circ F_T = \operatorname{Free}^T$ and $\operatorname{Forg}^T\circ K_T = U_T$ is called the \textit{comparison functor}, and it coincides with the embedding $K$ of $\cA_T$ into $\cA^T$ mentioned above.

\medskip

We are interested in the case when the comparison functor is an equivalence. Here, the category $\mathsf{Adj}_T$ has only one object, up to isomorphism, thus justifying the following terminology.

\begin{definition}
    A monad $T:\cA\to\cA$ is said be \textit{adjunction-trivial} if $\cA_T\simeq \cA^T$. 
\end{definition}

\subsection{Monadic division algebras}\label{subsec:MonDA} Let $\cC$ be a strict monoidal category. For $A \in \mathsf{Alg}(\cC)$, we get  monads  $\left((A\otimes -), \, m_A\otimes\id_{(-)},\, u_A\otimes\id_{(-)}\right)$ and $\left((-\otimes A),\, \id_{(-)}\otimes m_A, \,\id_{(-)}\otimes u_A\right)$ on $\cC$. To be consistent with the exclusion of the zero algebra as a division algebra in abelian monoidal categories, via  Lemma~\ref{lem:zeromod}, we will consider the condition below:

\smallskip

\begin{center}
 $A \in \mathsf{Alg}(\cC)$ satisfies that  $A$-$\mathsf{Mod}(\cC)$ and $\mathsf{Mod}$-$A(\cC)$ both have  more than one isoclass of objects. \hfill ($\star$)
\end{center}

\begin{definition}\label{def:monadic}
An algebra $A$ in $\cC$ subject to $(\star)$ is called a \textit{left (resp., right) monadic division algebra} if the monad $(A\otimes -)$ (resp., $(- \otimes A)$) on $\cC$ is adjunction-trivial.
The full subcategory of $\Alg(\cC)$ on these algebras is denoted by $\ell.\mathsf{MonDivAlg}(\cC)$ (resp.,  $r.\mathsf{MonDivAlg}(\cC)$).
\end{definition}

Note that essential division algebras can also be defined in $\cC$ by replacing the non-zero condition on $A$ with ($\star$). The connection to monadic division algebras in $\cC$ is given below.

\pagebreak 

\begin{proposition}\label{Prop:MonEss}
Take $A\in\Alg(\cC)$ subject to $(\star)$. Then, we have that 
$A\in \ell./r.\mathsf{MonDivAlg}(\cC)$ if and only if $A\in \ell./r.\mathsf{EssDivAlg}(\cC)$.  
\end{proposition}

\begin{proof}
This follows as $\cC^{(A\otimes -)}\simeq A\text{-}\Mod(\cC)$ and $\cC^{(-\otimes A)}\simeq \Mod\text{-}A(\cC)$, and because
under these equivalences,  $\cC_{(A\otimes -)}$ and $\cC_{(-\otimes A)}$ are the left and right free $A$-modules in $\cC$, respectively. 
\end{proof}

\begin{example} \label{ex:monadVec} For the monoidal category of $\Bbbk$-vector spaces, $(\mathsf{Vec}, \otimes_\Bbbk, \Bbbk)$, with $A \in \mathsf{Alg}(\mathsf{Vec})$, consider the monad  $(- \otimes_\Bbbk A)$  on $\mathsf{Vec}$. Then, $\mathsf{Vec}^{(- \otimes_\Bbbk A)} \simeq {\mathsf{Vec}}_{(- \otimes_\Bbbk A)}$  if and only if every right $A$-module over $\Bbbk$ is free, which happens precisely when $A$ is a division algebra over~$\Bbbk$. So, monadic division algebras in $\mathsf{Vec}$ again recover division algebras over $\Bbbk$.
\end{example}

\subsection{Examples of monadic division algebras}\label{subsec:MontoEssDA}

To construct more examples of monadic division algebras, we use monads that satisfy the following property.

\begin{definition}\label{def:strongMon}
    A monad $(T, \mu, \eta)$ on $\cC$ is \textit{left strong} if it is equipped with a natural transformation $\theta:=\{\theta_{X,Y}:X\otimes T(Y)\to T(X\otimes Y)\}_{X,Y\in\cC}$  ({\it left strength}) such that for all $X,Y,Z\in\cC$:

    \begin{enumerate}\renewcommand{\labelenumi}
    {\textnormal{(\roman{enumi})}}
        \item $\theta_{X,Y\otimes Z}\circ(\id_X\otimes\theta_{Y,Z}) = \theta_{X\otimes Y,Z}$;
        \qquad (iii)   $\theta_{X,Y}\circ(\id_X\otimes\mu_Y) = \mu_{X\otimes Y}\circ T(\theta_{X,Y})\circ\theta_{X,T(Y)}$;
        \smallskip
        \item $\theta_{\one,X} = \id_{T(X)}$; 
        \qquad \qquad  \qquad \qquad  \qquad   (iv)  $\theta_{X,Y}\circ (\id_X\otimes\eta_Y) = \eta_{X\otimes Y}$.
    \end{enumerate}

\noindent A left strong monad $(T,\mu, \eta,\theta)$ is said to be \textit{left very strong} if $\theta$ is a natural isomorphism. 
\textit{Right strong} and \textit{right very strong} monads on $\cC$ are defined analogously.
\end{definition}

The next result is straightforward to verify.

\begin{lemma}\label{lem:Monad to Alg}
If $(T, \mu, \eta, \theta)$ is left (resp., right) strong, then $T(\one) \in\mathsf{Alg}(\cC)$. Here, $m_{T(\one)}:= \mu_{\one} \, \theta_{T(\one), \one}$ (resp., $\mu_\one \,\theta_{\one,T(\one)}$), and $u_{T(\one)}:= \eta_{\one}$.   \qed
\end{lemma}

\begin{proposition}\label{Prop:adjtriv to Ess} Let $T$ be a monad on a strict monoidal category $\cC$.
    \begin{enumerate}\renewcommand{\labelenumi}{\textnormal{(\roman{enumi})}}
        \item If $T$ is left very strong, then $T$ is adjunction-trivial if and only if $T(\one)\in r.\mathsf{MonDivAlg}(\cC)$.
        \smallskip
        \item If $T$ is right very strong, then $T$ is adjunction-trivial if and only if $T(\one)\in \ell.\mathsf{MonDivAlg}(\cC)$.
    \end{enumerate}
\end{proposition}

\begin{proof}
In both cases, $T(\one)$ is an algebra in $\cC$ by Lemma~\ref{lem:Monad to Alg}.
Next, let $T$ be left very strong, with left strength $\theta$. Then, $T\cong (-\otimes T(\one))$ via the natural isomorphism $\theta_{-,\one}^{-1}$. It follows that $T$ is adjunction-trivial if and only if $(-\otimes T(\one))$ is adjunction-trivial, if and only if $T(\one)\in r.\mathsf{MonDivAlg}(\cC)$ by definition.    
Similarly, when $T$ is right very strong, $T\cong (T(\one)\otimes -)$, and part (ii) holds.
\end{proof}

\begin{example} Continuing Example~\ref{ex:monadVec} for $\cC = \mathsf{Vec}$ and $A \in \mathsf{Alg}(\mathsf{Vec})$,  we have that  $(- \otimes_\Bbbk A)$ is a left very strong monad on $\mathsf{Vec}$ with strength being the associativity of $\otimes_\Bbbk$. Now, $\mathsf{Vec}^{(- \otimes_\Bbbk A)} \simeq {\mathsf{Vec}}_{(- \otimes_\Bbbk A)}$  if and only if $\Bbbk \otimes_\Bbbk A \cong A$ is a right monadic division algebra in $\mathsf{Vec}$ (which happens precisely when $A$ is a division algebra over $\Bbbk$).
\end{example}

\begin{example}\label{ex4.2}
Consider the maybe monad $T$ on  $(\mathsf{Set},\sqcup,\emptyset)$, given by $T(-) := (-\sqcup\{*\})$, where $\sqcup$ is disjoint union. 
See \cite[Examples~5.1.4(i) and~5.3.2]{Riehl} for details; in particular, it is adjunction trivial. Also, $T$ is left very strong by the associativity of $\sqcup$.
Proposition~\ref{Prop:adjtriv to Ess} implies that $T(\emptyset)  \cong\{*\}$ is a right monadic division algebra in $(\mathsf{Set},\sqcup,\emptyset)$. Using that $X\sqcup Y\cong Y\sqcup X$ for $X,Y\in\mathsf{Set}$, we see that $T$ is right very strong. Hence, $\{*\}$ is also a left monadic division algebra in $(\mathsf{Set},\sqcup,\emptyset)$.
\end{example}

\begin{example} \label{ex:notverystrong}
Consider the free vector space monad $T$ on $(\mathsf{Set}, \times, \{\ast\})$, given by $T(X):= \Bbbk^X$, consisting of finitely supported functions $f: X \to \Bbbk$. See  \cite[Example 5.1.4(iii)]{Riehl}. We obtain that $\mathsf{Set}^T  \simeq \mathsf{Vec} \simeq {\mathsf{Set}}_T$. However,  $T$ is not left very strong as, in general, $X \times \Bbbk^Y \not \cong \Bbbk^{X \times Y}$ as sets. So, we cannot use Proposition~\ref{Prop:adjtriv to Ess} to get a left monadic division algebra in $\mathsf{Set}$. Still, see $\S$\ref{sec:divmonad} below.
\end{example}


\section{Discussion} \label{sec:directions}

We briefly discuss here potential research directions that may be of interest to the reader. 

\subsection{On division monads} \label{sec:divmonad}
One may want to refer to a monad $T$ on $\cA$ as a ``division monad'' when $\cA_T\simeq\cA^T$, instead of calling such monads adjunction-trivial. This would include monads that are not necessarily very strong, such as in Example~\ref{ex:notverystrong}. We inquire whether the scarcity of these types of monads mirrors the scarcity of division $\Bbbk$-algebras among the collection of $\Bbbk$-algebras. 




\subsection{On structural results for algebras in monoidal categories}

There are several classical results using  division $\Bbbk$-algebras that could be expanded to  general monoidal settings, e.g., Artin-Wedderburn Theorem.
Moreover, the classification of division algebras in various monoidal settings is open. For example, we expect an analogue of Frobenius's Theorem  (i.e., the only finite-dimensional division algebra over an algebraically closed field is the field itself) to hold in  finite monoidal settings.

\subsection{On the essential condition versus the simplistic condition} If one uses simplistic division algebras as  done in previous works (e.g., in \cite{GrossmanSnyder,Grossman,KongZheng}), then the supply of division algebras may be too abundant to make satisfactory progress. For instance, any simple module over a finite group $G$ yields a simplistic division algebra in the monoidal category of $G$-modules [Example~\ref{ex:grouptype}]. We propose it is that better to use the more restrictive class of essential/monadic division algebras to examine pertinent results for algebras in monoidal categories.

\subsection{On the left versus right division algebra conditions} In Proposition~\ref{prop:LiffR}, we proved that a left simplistic (resp., essential) division algebra in a pivotal multifusion category $\cC$ is a right simplistic (resp., essential) division algebra in $\cC$, and vice versa. After the appearance of our article, it was shown in work of Nakamura, Shibara, and Shimuzu that the result in the simplistic case holds when $\cC$ is a finite tensor category \cite{NSS2025}. There, the more common terminology, {\it left/right simple algebra in $\cC$}, is used instead of our terminology here. We all expect that such ``Left $\Leftrightarrow$ Right'' results hold in more general monoidal settings \cite{KS-JKCW}.


\bibliography{DivAlg}

\begin{thebibliography}{EGNO15}

\bibitem[BD12]{BookerDavydov}
Thomas Booker and Alexei Davydov.
\newblock Commutative algebras in {F}ibonacci categories.
\newblock {\em J. Algebra}, 355:176--204, 2012.

\bibitem[EGNO15]{EGNO}
Pavel Etingof, Shlomo Gelaki, Dmitri Nikshych, and Victor Ostrik.
\newblock {\em Tensor categories}, volume 205 of {\em Mathematical Surveys and
  Monographs}.
\newblock American Mathematical Society, Providence, RI, 2015.

\bibitem[ENO05]{EtingofNikshychOstrik}
Pavel Etingof, Dmitri Nikshych, and Viktor Ostrik.
\newblock On fusion categories.
\newblock {\em Ann. of Math. (2)}, 162(2):581--642, 2005.

\bibitem[Gro19]{Grossman}
Pinhas Grossman.
\newblock Fusion categories associated to subfactors with index {$3+\sqrt 5$}.
\newblock {\em Indiana Univ. Math. J.}, 68(4):1277--1325, 2019.

\bibitem[GS16]{GrossmanSnyder}
Pinhas Grossman and Noah Snyder.
\newblock The {B}rauer-{P}icard group of the {A}saeda-{H}aagerup fusion
  categories.
\newblock {\em Trans. Amer. Math. Soc.}, 368(4):2289--2331, 2016.

\bibitem[KZ19]{KongZheng}
Liang Kong and Hao Zheng.
\newblock Semisimple and separable algebras in multi-fusion categories.
\newblock {\em \textnormal{arXiv preprint arXiv:1706.06904v2}}, 2019.

\bibitem[Mog91]{Moggi}
Eugenio Moggi.
\newblock Notions of computation and monads.
\newblock {\em Information and computation}, 93(1):55--92, 1991.

\bibitem[MU22]{McDermottUustalu}
Dylan McDermott and Tarmo Uustalu.
\newblock What makes a strong monad?
\newblock In {\em Proceedings---{N}inth {W}orkshop on {M}athematically
  {S}tructured {F}unctional {P}rogramming}, volume 360 of {\em Electron. Proc.
  Theor. Comput. Sci. (EPTCS)}, pages 113--133. EPTCS, 2022.

\bibitem[NSS25]{NSS2025}
Daisuke Nakamura, Taiki Shibata, and Kenichi Shimizu.
\newblock Exact module categories over {R}ep$(u_q(\mathfrak{sl}_2))$.
\newblock {\em \textnormal{arXiv preprint arXiv:2503.21265}}, 2025.

\bibitem[Ost03]{Ostrik}
Victor Ostrik.
\newblock Module categories, weak {H}opf algebras and modular invariants.
\newblock {\em Transform. Groups}, 8(2):177--206, 2003.

\bibitem[Rie17]{Riehl}
Emily Riehl.
\newblock {\em Category theory in context}.
\newblock Courier Dover Publications, 2017.

\bibitem[Shi25]{KS-JKCW}
Kenichi Shimizu.
\newblock Private communication, February 2025.

\bibitem[Wal24]{Walton}
Chelsea Walton.
\newblock {\em Symmetries of {A}lgebras}, volume~1.
\newblock 619 Wreath Publishing, Oklahoma City, OK, 2024.

\end{thebibliography}
\bibliographystyle{alpha}

\end{document}